\documentclass[11pt]{article}
\usepackage{fullpage}
\usepackage{amsmath,amssymb,amsthm,mathtools}
\usepackage{microtype,booktabs,tabularx,array,enumitem,needspace}
\usepackage{algorithm,algpseudocode}
\usepackage[numbers,sort&compress]{natbib}
\usepackage{hyperref}
\usepackage[nameinlink,noabbrev]{cleveref}

\newcommand{\bbE}{\mathbb{E}}

\newcommand{\bbP}{\mathbb{P}}

\newcommand{\bbR}{\mathbb{R}}

\newcommand{\cO}{\mathcal{O}}

\newcommand{\rmd}{\mathrm{d}}

\renewcommand{\leq}{\leqslant}
\renewcommand{\geq}{\geqslant}

\numberwithin{equation}{section}
\newtheorem{theorem}{Theorem}[section]
\newtheorem{lemma}{Lemma}

\newtheorem{corollary}{Corollary}
\theoremstyle{definition}
\newtheorem{definition}{Definition}
\newtheorem{assumption}{Assumption}
\theoremstyle{remark}

\crefname{theorem}{Theorem}{Theorems}
\crefname{lemma}{Lemma}{Lemmas}
\crefname{proposition}{Proposition}{Propositions}
\crefname{corollary}{Corollary}{Corollaries}
\crefname{definition}{Definition}{Definitions}
\crefname{assumption}{Assumption}{Assumptions}
\crefname{algorithm}{Algorithm}{Algorithms}
\setlist{itemsep=2pt,topsep=4pt}
\newcommand{\eps}{\varepsilon}
\newcommand{\norm}[1]{\left\lVert#1\right\rVert}

\newcommand{\lammin}{\lambda_{\min}}
\title{\textbf{Matching Complexity for Second-Order Stationary Point with Sample Reuse}}
\author{}
\date{}

\newcommand{\authornames}{Qian Chen, Wendao Wu, Haihan Zhang, Chenheng Zhang, Chunyuan Zheng, Cong Fang, Haoxuan Li, Zhouchen Lin}
\newcommand{\affiliation}{Peking University}

\newcommand{\mail}[1]{\href{mailto:#1}{\texttt{#1}}}
\newif\ifanonymous
\anonymousfalse
\ifanonymous
\author{Anonymous Authors}
\hypersetup{pdfauthor={Anonymous Authors}}
\else
\author{Qian Chen$^{1,*}$\quad Wendao Wu$^{1,*}$\quad Haihan Zhang$^{1,*}$ \\[0.20em]
Chenheng Zhang$^{1,*}$\quad Chunyuan Zheng$^1$\\[0.20em]
Cong Fang$^{1,\dagger}$\quad Haoxuan Li$^{1,\dagger}$\quad Zhouchen Lin$^{1,\dagger}$\\[0.60em]
\small $^1$\affiliation\\[0.35em]
\footnotesize\mail{colombo.2026@outlook.com}\\[0.35em]
\small $^*$Equal contribution.\quad $^\dagger$Corresponding authors.}
\hypersetup{pdfauthor={\authornames}}
\fi

\newcommand{\ResearchAgentSystem}{Colombo, our laboratory's internal auto-research system}

\begin{document}
\maketitle
\begin{abstract}
We consider stochastic nonconvex optimization with sample-wise Lipschitz gradients and Hessians, and reusable opaque sample handles. Gradient and complete-Hessian evaluations are charged separately, including both endpoints of a same-sample difference. For fixed positive problem parameters and in the high-dimensional worst case, we prove that the minimax total derivative-query complexity of finding an $(\varepsilon,\gamma)$-second-order stationary point is
\begin{equation*}
\widetilde\Theta(\varepsilon^{-3}+\gamma^{-3}).
\end{equation*}
The upper bound combines curvature-first randomly signed negative-curvature steps with binary-ancestor derivative estimators; a truncated increment estimate charges predictable drift linearly and martingale motion quadratically. The curvature lower bound uses a deterministic smooth bounded-gradient chain, a radial barrier, and a random rotation. The new first-order lower bound is specifically robust to persistent sample reuse and to complete-Hessian access. It uses a slowly parameterized conservative sign chain, a third-order reverse stick-breaking localization inequality, and stateful relative-entropy accounting for reusable handles. Thus the total rate is near-optimally matched for every relation between $\varepsilon$ and $\gamma$.
\end{abstract}


\par\medskip
{\small
\begin{samepage}
\noindent\textbf{AI Usage.}
Nearly the entire research pipeline for this paper was carried out by
\ResearchAgentSystem{}, powered by GPT-5.6 Sol. The system also conducted a
Lean-backed article audit of the resulting manuscript. The authors subsequently reviewed and approved the
mathematical claims, presentation, and formal artifacts, and take
responsibility for the final manuscript. The complete Lean audit report and the system's technical report will be made public at a later date.

\paragraph{Publication Intent and Recognition of Independent Work.}
This preprint is released as a research output of our study of
automated scientific research, for public scrutiny and discussion.
As we do not intend to submit this manuscript to a journal or
conference, it has not undergone the full editorial refinement
and internal review we would normally require for submission.
The authors remain responsible for its contents and for correcting
any errors or inaccurate attribution.

We recognize that releasing these results may affect researchers
who have independently been pursuing the same questions, including
work that is not yet public.
If this applies to your work, we invite you to contact us at
\texttt{colombo.2026@outlook.com}.
A brief description of the overlap, together with any dated
materials you are comfortable sharing, would help us understand
its scope and chronology.
Such materials might include draft excerpts, presentation slides,
or revision records; a complete manuscript or unpublished proof
is not required to begin the conversation.

We will review these cases carefully and update citations,
acknowledgments, and accounts of independent or concurrent work
as appropriate, including revising priority claims where warranted.
Nonpublic materials will be treated confidentially and used only
to assess the overlap unless their authors authorize otherwise.
We will obtain explicit permission from the authors before
publicly discussing such materials or providing them to an
automated research system.
\par
\end{samepage}
}

\section{Introduction}\label{sec:intro}
We study the stochastic nonconvex optimization problem
\begin{equation*}
    \min_{x\in\bbR^d} f(x),\qquad f(x):=\bbE_{\xi}F(x;\xi),
\end{equation*}
where $\xi$ is a random sample, and $F(\cdot;\xi)$ is a twice continuously differentiable sample loss with Lipschitz gradient and Hessian. An oracle evaluates the gradient or complete Hessian of a selected sample loss at a query point. Our goal is to find an $(\eps,\gamma)$-second-order stationary point (SOSP), namely a point $x$ satisfying
\begin{equation*}
    \norm{\nabla f(x)}\le\eps,\qquad
    \lambda_{\min}(\nabla^2f(x))\ge-\gamma.
\end{equation*}
These conditions provide a tractable approximate local optimality criterion: they exclude a large gradient and a direction of sufficiently negative curvature~\cite{nesterov2006cubic}. Their significance is particularly clear under a quantitative strict-saddle condition, under which every point outside neighborhoods of local minima has either a gradient bounded away from zero or a negative Hessian eigenvalue bounded away from zero. Sufficiently accurate SOSPs must then lie near local minima~\cite[Definition~5]{ge2015escaping}. We do not impose this additional landscape condition, and an SOSP need not be a local minimum for a general nonconvex objective. Keeping $\eps$ and $\gamma$ independent distinguishes the cost of reducing the gradient from that of controlling negative curvature.

Existing complexity bounds depend on both the requested tolerances and the available sample access; Table~\ref{tab:comparison} summarizes the relevant guarantees. SPIDER achieves $\widetilde{\cO}(\eps^{-3}+\eps^{-2}\gamma^{-2}+\gamma^{-5})$ stochastic gradient evaluations using shared samples~\cite[Corollary~7]{fang2018spider}. Arjevani et al.~\cite[Theorem~5]{arjevani2020second} obtain the same tolerance dependence with fresh gradients and Hessian-vector products (HVPs), and Wu et al.~\cite{wu2026second} remove the mixed term to obtain $\widetilde{\cO}(\eps^{-3}+\gamma^{-5})$ in a fresh common-random-loss model. Under uniform sample-wise smoothness, Carme~\cite[Theorem~1]{shi2026curvature} attains expected gradient and Hessian costs of $\widetilde{\cO}(\eps^{-3}+\eps^{-2}\gamma^{-2}+\gamma^{-3})$ and $\widetilde{\cO}(\eps^{-2}+\gamma^{-4})$, respectively. These results already yield the $\widetilde{\cO}(\eps^{-3})$ total rate at the standard coupled target $\gamma\asymp\sqrt\eps$, but leave a larger cost when a more stringent curvature tolerance is prescribed.

The distinction between streaming access and sample reuse is central to this question. In the fresh, single-point model, each oracle call draws an independent sample that is discarded after one evaluation~\cite{arjevani2020second,wu2026second}. A multi-point query can instead evaluate a common sample at several chosen points, as in variance-reduced derivative differences~\cite{arjevani2020second,fang2018spider}. We allow the stronger persistent access in which the algorithm saves an opaque sample handle and adaptively revisits its loss throughout the run. This models the ability to retain sampled data and reevaluate their losses, while shared-sample differences already play a central role in SPIDER and recursive cubic regularization~\cite{fang2018spider,zhou2020stochastic}. Under sample-wise smoothness, their centered errors shrink with the distance between the query points, enabling derivative tracking; Lemma~\ref{lem:accuracy} makes this relation precise. Every evaluated endpoint remains charged. The question is therefore how much this access can reduce the two-tolerance complexity, and whether a lower bound survives arbitrary repeated interrogation of saved samples.

Our results characterize the total derivative-query complexity in this model. Let $N_g$ and $N_H$ denote the numbers of stochastic gradient and complete-Hessian evaluations, respectively. For the upper bound, Theorem~\ref{thm:upper} constructs an algorithm with deterministic query caps and success probability at least $2/3$ such that
\begin{equation*}
    N_g=\widetilde{\cO}(\eps^{-3}+\eps^{-2}\gamma^{-1}),\qquad
    N_H=\widetilde{\cO}(\eps^{-2}+\eps^{-1}\gamma^{-2}+\gamma^{-3}).
\end{equation*}
Consequently, $N_g+N_H=\widetilde{\cO}(\eps^{-3}+\gamma^{-3})$. For the lower bound, Theorem~\ref{thm:lower} constructs admissible instances requiring $\widetilde\Omega(\eps^{-3}+\gamma^{-3})$ total calls from any randomized adaptive algorithm, even with persistent sample reuse and joint value, gradient, and complete-Hessian responses. Together, these results imply
\begin{equation}\label{eq:intro-total}
    Q_{\mathrm{tot}}^\star(\eps,\gamma)
    =\widetilde\Theta(\eps^{-3}+\gamma^{-3}).
\end{equation}
Here the problem parameters are fixed and positive, and tildes suppress logarithmic factors. The upper bound holds in every finite dimension; the matching lower bound permits dimension to grow polynomially with inverse tolerances. This is a characterization of total evaluations, while the optimal tradeoff between the two separate budgets remains open.

For the upper bound, Algorithm~\ref{alg:curvature-first} tests the Hessian before updating the gradient. If it detects negative curvature, it takes a step with an independent random sign; otherwise it estimates the gradient and either terminates or takes a normalized gradient step. Two binary-ancestor estimators cache shared-sample differences, so each derivative estimate contains only logarithmically many estimation errors. The central cost estimate, Lemma~\ref{lem:charge}, separates predictable gradient displacements from the martingale displacements of randomly signed curvature steps. It charges the former by total length and the latter by the sum of squared lengths. Expected descent bounds these quantities by $\cO(\Delta/\eps)$ and $\cO(\Delta/\gamma)$, respectively. Applying the estimate to both derivative fields, with gradient updates restricted to the subsequence that passes the curvature test, yields the stated budgets. Weighted arithmetic--geometric mean inequalities absorb the mixed terms into $\eps^{-3}+\gamma^{-3}$; Section~\ref{sec:upper} contains the proof.

For the lower bound, we construct two distinct families. The first-order family encodes independent hidden sign vectors in a chain of stages, each of which must be activated before an $\eps$-stationary point can be reached. With a logarithmic block size $D$, there are $\Theta(\Delta/(\eps D^{1/4}))$ stages, and a relative-entropy argument forces $\Omega(s_\star^2/\eps^2)$ charged observations per stage, where $s_\star=\min\{1,\sigma_g,\sigma_H,L_1,L_2,G\}$. This yields $\Omega(\Delta s_\star^2/(\eps^3D^{1/4}))$ total calls; Section~\ref{sec:firstorder-lb} proves the stagewise information bound. The curvature family is deterministic. A smooth chain with $\Theta(\gamma^{-3})$ links retains a negative eigenvalue below $-\gamma$ until all links have been completed. Flat gates, a random rotation, and a bounded-gradient radial wall prevent unrestricted queries or distant outputs from bypassing the chain, even when a full Hessian is returned (Section~\ref{sec:curvature-lb}). Taking the maximum of these two endpoint lower bounds establishes their sum up to constants and logarithms.

The additional work required by sample reuse concerns both estimation and information control. Dyadic tracking, the separation of drift from randomly signed curvature motion, and conservative scalar noise localization have direct predecessors in Wu et al.~\cite{wu2026second}; the present upper bound extends the path-cost argument to shared-sample Hessian differences under uniform sample-wise regularity. For the lower bound, a fresh-response entropy calculation cannot be charged again whenever an old sample is queried. We represent each sample by a fixed coefficient table and strengthen the oracle to reveal the relevant block of that table on its first use. Later queries involving the same handle and block are deterministic functions of the revealed coefficients, and contribute no additional conditional relative entropy; a new block of the same handle is accounted for separately. We also extend the smooth localization estimate through third order and choose the coordinate scale $D^{1/4}$ so that every sampled scalar loss has a uniformly Lipschitz Hessian, independently of chain length. These two steps preserve the first-order obstruction under persistent reuse, at the cost of only a logarithmic factor. Our upper bound itself needs only two evaluations of each newly drawn sample, whereas the lower bound covers unlimited persistent reuse.

\begin{table}[t]
\centering\small
\begin{tabular}{lll}
\toprule
Algorithm / reference & Bound & Oracle\\
\midrule
SPIDER-SFO$^+$~\cite{fang2018spider}
& $\widetilde{\cO}(\eps^{-3}+\eps^{-2}\gamma^{-2}+\gamma^{-5})$
& FO; two-point\\
\midrule
\citet[Alg.~4]{arjevani2020second}
& $\widetilde{\cO}(\eps^{-3}+\eps^{-2}\gamma^{-2}+\gamma^{-5})$
& FO + HVP; one-point\\
\midrule
\citet[Thm.~3.3]{wu2026second} (upper)
& $\widetilde{\cO}(\eps^{-3}+\gamma^{-5})$
& FO + HVP; one-point\\
\citet[Thm.~3.4]{wu2026second} (lower)
& $\widetilde\Omega(\eps^{-3}+\gamma^{-5})$
& Joint SO; one-point\\
\midrule
Carme~\cite{shi2026curvature}
& $\begin{aligned} &\bbE N_g=\widetilde{\cO}(\eps^{-3}+\eps^{-2}\gamma^{-2}+\gamma^{-3}),\\ &\bbE N_H=\widetilde{\cO}(\eps^{-2}+\gamma^{-4}) \end{aligned}$
& SO; two-point\\
\midrule
Algorithm~\ref{alg:curvature-first} (ours, upper)
& $\begin{aligned} &N_g=\widetilde{\cO}(\eps^{-3}+\eps^{-2}\gamma^{-1}),\\ &N_H=\widetilde{\cO}(\eps^{-2}+\eps^{-1}\gamma^{-2}+\gamma^{-3}) \end{aligned}$
& SO; two-point\\
Theorem~\ref{thm:lower} (ours, lower)
& $\widetilde\Omega(\eps^{-3}+\gamma^{-3})$
& Joint SO; arbitrary reuse\\
\bottomrule
\end{tabular}

\smallskip
\begin{minipage}{\textwidth}
\footnotesize
\textit{Oracle notation.} FO: stochastic gradient; HVP: Hessian-vector product; SO: gradient and complete Hessian, charged separately; Joint SO: a joint second-order oracle returning the value, gradient, and complete Hessian of the same sample loss at the same point, charged as one call. One-point: a fresh sample per call; two-point: a fresh sample evaluated at two prescribed points, with both endpoints charged; arbitrary reuse: unlimited adaptive revisits to stored samples. $N_g$ and $N_H$ count gradient and complete-Hessian evaluations.
\end{minipage}
\caption{Two-tolerance query bounds for fixed positive scale parameters. Tildes suppress logarithmic factors, including dimension factors; constant tolerance rescalings do not affect the displayed orders. Assumptions differ across papers; see Section~\ref{sec:related}.}
\label{tab:comparison}
\end{table}
\paragraph{Notation.}
We use $\norm{\cdot}$ for the Euclidean norm of vectors and the induced operator norm of matrices and multilinear maps; an explicit subscript $p$ denotes the vector $\ell_p$ norm. The symbols $\langle\cdot,\cdot\rangle$ and $(\cdot)^\top$ denote the Euclidean inner product and transpose, respectively. For a symmetric matrix $A$, $\lammin(A)$ denotes its smallest eigenvalue. We write $D^k h$ for the $k$th derivative of $h$ and $\operatorname{Lip}(A)$ for the Lipschitz constant of a map $A$ in these norms. The symbols $\bbE$ and $\bbP$ denote expectation and probability. We use $\cO$, $\Omega$, and $\Theta$ for asymptotic bounds; their tilded versions additionally suppress logarithmic factors, including those in the dimension and inverse tolerances.

\section{Related Work}\label{sec:related}
\paragraph{Upper bounds for SOSP.}
Cubic regularization provides a classical deterministic approach to second-order stationarity~\cite{nesterov2006cubic}. SPIDER-SFO$^+$~\cite[Corollary~7]{fang2018spider} uses shared-sample gradient differences to attain $\widetilde{\cO}(\eps^{-3}+\eps^{-2}\gamma^{-2}+\gamma^{-5})$ queries. Arjevani et al.~\cite[Theorem~5]{arjevani2020second} attain the same tolerance dependence with fresh gradients and HVPs from general stochastic derivative estimators. Their theorem requires sufficiently small tolerances with $\gamma\le\sqrt{L_2\eps}$; a looser curvature target can be met by running it at $\sqrt{L_2\eps}$. Wu et al.~\cite[Theorem~3.3]{wu2026second} improve the fresh common-random-loss bound to $\widetilde{\cO}(\eps^{-3}+\gamma^{-5})$ under population Hessian regularity, using dyadic gradient tracking and separate accounting for drift and random curvature motion. With reusable samples, Zhou and Gu~\cite{zhou2020stochastic} develop recursive gradient and Hessian estimators for finite-sum cubic regularization at the coupled target $(\eps,\sqrt\eps)$, while Chayti, Doikov, and Jaggi~\cite{chayti2023unified} unify the analysis of stochastic and variance-reduced cubic Newton methods. Carme~\cite[Theorem~1]{shi2026curvature} treats independent tolerances using two-point gradient and complete-Hessian differences under bounded derivative noise and uniform sample-wise smoothness. Its separate derivative budgets, together with the preceding bounds, appear in Table~\ref{tab:comparison}. Under the assumptions in Section~\ref{sec:setup}, our curvature-first analysis reduces the mixed gradient term and the curvature cost of Hessian estimation; Theorem~\ref{thm:upper} states the resulting budgets.

\paragraph{Lower bounds for SOSP.}
Carmon et al.~\cite{carmon2020lower} establish an $\Omega(\eps^{-3/2})$ lower bound for finding a point with $\norm{\nabla f(x)}\le\eps$ on deterministic functions with Lipschitz Hessians, even for randomized algorithms with exact derivative information. In the stochastic setting, Arjevani et al.~\cite[Theorems~4 and~6]{arjevani2020second} prove an $\Omega(\eps^{-3}+\gamma^{-5})$ lower bound for zero-respecting second-order methods at fixed positive scale parameters. Their noiseless curvature bound specializes to $\Omega(\gamma^{-3})$ for zero-respecting second-order methods~\cite[Theorem~6]{arjevani2020second}. Their stochastic oracle permits general unbiased derivative estimators, without requiring derivatives of a common scalar sample loss. Their first-order lower bound also extends to simultaneous multi-point queries~\cite[Section~3.2]{arjevani2020second}, a model distinct from arbitrary persistent reuse. Wu et al.~\cite[Theorem~3.4]{wu2026second} establish $\widetilde\Omega(\eps^{-3}+\gamma^{-5})$ for arbitrary randomized adaptive algorithms with fresh joint value, gradient, and complete-Hessian responses. This model additionally assumes bounded function-value variance and does not prescribe a uniform sample Hessian Lipschitz constant. These results do not directly establish a lower bound under our combination of persistent reuse and uniform sample-wise smoothness. Theorem~\ref{thm:lower} addresses this setting and proves $\widetilde\Omega(\eps^{-3}+\gamma^{-3})$ total queries, including for the stronger joint oracle.

\paragraph{Hidden chains and concentration tools.}
Hidden chains and random rotations are central to the lower bounds of Carmon et al.~\cite{carmon2020lower}; Arjevani et al.~\cite[Appendix~H]{arjevani2020second} develop stochastic derivative constructions based on probabilistic zero-chains and negative-curvature obstructions. Wu et al.~\cite{wu2026second} localize noise in a scalar sample loss so that its value, gradient, and Hessian form a consistent joint response. Our curvature construction combines a smooth hidden chain with a bounded-gradient radial extension to control unrestricted queries and distant outputs (Section~\ref{sec:curvature-lb}). For the stochastic construction, we extend scalar noise localization through third order to enforce uniform sample Hessian regularity, and account for information only when a stored sample's relevant coefficient block is first revealed. Lemmas~\ref{lem:thirdstick} and~\ref{lem:persistent-info} establish these two properties, including under repeated adaptive queries to the same sample.

\section{Problem Formulation}\label{sec:setup}
Let $\xi$ be a random sample, and let $F(x;\xi)$ be a measurable scalar sample loss on $\bbR^d$. We consider
\begin{equation}\label{eq:problem}
    \min_{x\in\bbR^d} f(x),\qquad
    f(x):=\bbE_{\xi}F(x;\xi),
\end{equation}
where the expectation is finite at every $x$. The algorithm starts from a specified point $x_0\in\bbR^d$. The tolerances $\eps,\gamma>0$ are independent.

\begin{definition}[Second-order stationarity point]\label{def:sosp}
A point $x\in\bbR^d$ is an $(\eps,\gamma)$-second-order stationary point of $f$ if
\begin{equation*}
    \norm{\nabla f(x)}\le\eps,\qquad
    \lammin(\nabla^2f(x))\ge-\gamma.
\end{equation*}
\end{definition}
We seek a randomized algorithm whose output $\hat x$ satisfies
\begin{equation}\label{eq:success}
    \bbP\left(\norm{\nabla f(\hat x)}\leq\eps,\quad
    \lammin(\nabla^2f(\hat x))\geq-\gamma\right)\geq\frac23
\end{equation}
on every admissible instance. The probability includes both the algorithm's internal randomness and its sampled losses. A failure symbol $\bot$ is permitted, but counts as failure in~\eqref{eq:success}.

\begin{definition}[Reusable-sample derivative oracle]\label{def:oracle}
An instance specifies a scalar random loss $F(\cdot;\xi)$. The algorithm can request a fresh opaque handle $r$. Conditional on the fixed instance, the oracle assigns to $r$ a copy $\xi_r$ of $\xi$, independent of the previous history and of the algorithm's random seed. The handle reveals neither $\xi_r$ nor a description of $F(\cdot;\xi_r)$. At any later time, a query selects a point $x\in\bbR^d$, a saved or newly requested handle $r$, and one of the following response types:
\begin{equation*}
    \mathcal O_g(x,r)=\nabla F(x;\xi_r),\qquad
    \mathcal O_H(x,r)=\nabla^2F(x;\xi_r).
\end{equation*}
The first returns the complete gradient and costs one gradient evaluation; the second returns the complete Hessian and costs one Hessian evaluation. A handle always retains its original sample, and may be reused arbitrarily many times at adaptively selected points. Requesting a handle supplies no numerical observation and incurs no additional charge beyond its subsequent evaluations.
\end{definition}
Let $N_g$ and $N_H$ be the numbers of gradient and Hessian evaluations, and let $Q=N_g+N_H$ be the total query count. Querying both derivatives at the same point costs two evaluations. A same-sample two-point difference also costs two evaluations of the corresponding type. Function values and third derivatives are unavailable in Definition~\ref{def:oracle}. Computation on observed responses, including matrix-vector products and eigenvector calculations, is free. All query and output maps are measurable functions of the preceding transcript and private randomness, and may select arbitrary points of $\bbR^d$; the output need not have been queried.

\subsection{Assumptions}\label{sec:assumptions}
We treat problem parameters $\Delta,L_1,L_2,\sigma_g,\sigma_H,G$ as constants. An admissible instance satisfies the following assumptions.
\begin{assumption}[Initial objective gap]\label{ass:gap}
The population objective is bounded below, and the initial point satisfies
\begin{equation*}
    \inf_{x\in\bbR^d}f(x)>-\infty,\qquad
    f(x_0)-\inf_{x\in\bbR^d}f(x)\le\Delta.
\end{equation*}
\end{assumption}
\begin{assumption}[Uniform sample smoothness]\label{ass:smoothness}
Almost surely in $\xi$, $F(\cdot;\xi)\in C^2(\bbR^d)$ and the following bounds hold simultaneously for all $x,y\in\bbR^d$:
\begin{equation*}
    \norm{\nabla F(x;\xi)-\nabla F(y;\xi)}\le L_1\norm{x-y},\qquad
    \norm{\nabla^2F(x;\xi)-\nabla^2F(y;\xi)}\le L_2\norm{x-y}.
\end{equation*}
\end{assumption}
\begin{assumption}[Unbiased derivatives]\label{ass:unbiased}
The population objective $f$ belongs to $C^2(\bbR^d)$, and differentiation commutes with expectation through order two:
\begin{equation*}
    \bbE_\xi\nabla F(x;\xi)=\nabla f(x),\qquad
    \bbE_\xi\nabla^2F(x;\xi)=\nabla^2f(x)
    \quad\text{for every }x\in\bbR^d.
\end{equation*}
\end{assumption}
\begin{assumption}[Uniform derivative noise bounds]\label{ass:noise}
Almost surely in $\xi$, the following bounds hold simultaneously for all $x\in\bbR^d$:
\begin{equation*}
    \norm{\nabla F(x;\xi)-\nabla f(x)}\le\sigma_g,\qquad
    \norm{\nabla^2F(x;\xi)-\nabla^2f(x)}\le\sigma_H.
\end{equation*}
\end{assumption}
\begin{assumption}[Bounded population gradient]\label{ass:population-gradient}
The population gradient is globally bounded:
\begin{equation*}
    \sup_{x\in\bbR^d}\norm{\nabla f(x)}\le G.
\end{equation*}
\end{assumption}
Assumptions~\ref{ass:smoothness} and~\ref{ass:unbiased}, together with Jensen's inequality, imply the same gradient and Hessian Lipschitz bounds for $f$. Assumption~\ref{ass:noise} is an almost-sure bound, stronger than a variance bound. No convexity or strict-saddle condition is required. All six public bounds remain fixed as $\eps$ and $\gamma$ vary.

For fixed $d$ and these bounds, let $Q_d^\star(\eps,\gamma)$ be the smallest deterministic total query cap of an algorithm satisfying~\eqref{eq:success} on every admissible instance under Definition~\ref{def:oracle}. The notation $Q_{\mathrm{tot}}^\star$ summarizes the upper bound for every finite $d$ and the lower bound in sufficiently large dimensions specified below. Tildes hide logarithms in $d$ and inverse tolerances, never polynomial dependence on them. Matching complexity is a high-dimensional worst-case statement; it does not assert the same lower bound in every fixed dimension or a dimension-independent supremum over arbitrarily large $d$.
\section{Main Results}\label{sec:results}
\subsection{Upper Bound}\label{sec:main-upper}
We first present the optimization algorithm. Its subroutine $\textsc{TreeEstimate}(A,x,\eta)$, specified in Algorithm~\ref{alg:tree-estimate} in Section~\ref{sec:tracking}, estimates the mean derivative field at $x$ with target error $\eta$. Each estimator retains its own checkpoint locations and cached estimates across calls: the first call averages fresh samples, and later calls add a same-sample difference to a cached ancestor estimate. The optimization algorithm maintains separate instances for gradients and Hessians, calibrated with $(\sigma,L)=(\sigma_g,L_1)$ and $(\sigma_H,L_2)$, respectively. The Hessian tree is updated at every iterate, while the gradient tree is updated only after the curvature test passes.

\begin{algorithm}[t]
\caption{Curvature-first optimization with reusable samples}\label{alg:curvature-first}
\begin{algorithmic}[1]
\Require Initial point $x_0$; tolerances $\eps,\gamma$; public problem bounds; iteration cap $T$ and query caps $Q_g,Q_H$.
\State Initialize two empty instances of $\textsc{TreeEstimate}$ (see Algorithm~\ref{alg:tree-estimate}); set $N_g=N_H=0$.
\State Use error targets $\eta_g=\eps/4$, $\eta_H=\gamma/4$ and failure probability $1/96$ per tree.
\State Charge each oracle evaluation to $N_g$ or $N_H$; return $\bot$ before exceeding its cap.
\For{$t=0,\ldots,T-1$}
    \State $H_t\gets\textsc{TreeEstimate}(\nabla^2F,x_t,\gamma/4)$.
    \If{$\lammin(H_t)<-3\gamma/4$}
        \State Choose a unit eigenvector $v_t$ associated with $\lammin(H_t)$.
        \State Draw $s_t$ uniformly from $\{-1,+1\}$, independently of the history.
        \State $x_{t+1}\gets x_t+s_t(\gamma/L_2)v_t$.
    \Else
        \State $g_t\gets\textsc{TreeEstimate}(\nabla F,x_t,\eps/4)$.
        \If{$\norm{g_t}\le3\eps/4$}
            \State \Return $x_t$.
        \EndIf
        \State $x_{t+1}\gets x_t-(\eps/(4L_1))g_t/\norm{g_t}$.
    \EndIf
\EndFor
\State \Return $\bot$.
\end{algorithmic}
\end{algorithm}
To state its guarantee with the problem parameters explicit, define
\begin{align}
    B_g={}&1+\frac{\Delta L_1}{\eps^2}+\frac{\sigma_g^2}{\eps^2}
    +\frac{\sigma_gL_1\Delta}{\eps^3}
    +\frac{L_1^2\Delta}{\eps^2\gamma},\label{eq:B_g}\\
    B_H={}&1+\frac{\Delta L_1}{\eps^2}+\frac{\sigma_H^2}{\gamma^2}
    +\frac{\sigma_HL_2\Delta}{\eps\gamma^2}
    +\frac{2L_2^2\Delta}{\gamma^3}.\label{eq:B_H}
\end{align}
The two contributions to the last term in $B_H$ come from the number of curvature checkpoints and from tracking the randomly signed curvature motion; Section~\ref{sec:upper-proof} derives them separately.

\begin{theorem}[Upper bound with deterministic query caps]\label{thm:upper}
Under Assumptions~\ref{ass:gap}--\ref{ass:population-gradient}, consider the oracle in Definition~\ref{def:oracle}. Let $d\ge1$ and $0<\eps,\gamma\le1$. There is a universal constant $C\ge1$ such that the following choices suffice for Algorithm~\ref{alg:curvature-first}:
\begin{equation*}
    T=1+\left\lceil48\Delta\left(\frac{32L_1}{3\eps^2}
    +\frac{12L_2^2}{\gamma^3}\right)\right\rceil,
    \qquad \mathcal L=1+\log\left(C(d+1)(T+1)\right),
\end{equation*}
\begin{equation*}
    Q_g=\left\lceil48C\mathcal L^4B_g\right\rceil,
    \qquad Q_H=\left\lceil48C\mathcal L^4B_H\right\rceil.
\end{equation*}
Use the batch sizes of Lemma~\ref{lem:accuracy}, with failure probability $1/96$ for each tree and the error targets in the algorithm. Then $N_g\le Q_g$ and $N_H\le Q_H$ on every execution, and its output satisfies
\begin{equation*}
    \bbP\left(\norm{\nabla f(\hat x)}\leq\eps,\quad
    \lammin(\nabla^2f(\hat x))\geq-\gamma\right)\geq\frac23.
\end{equation*}
For fixed positive problem parameters, the separate budgets simplify to
\begin{align}
    N_g&=\widetilde{\cO}(\eps^{-3}+\eps^{-2}\gamma^{-1}),\label{eq:Ng}\\
    N_H&=\widetilde{\cO}(\eps^{-2}+\eps^{-1}\gamma^{-2}+\gamma^{-3}),\label{eq:NH}
\end{align}
and hence $N_g+N_H=\widetilde{\cO}(\eps^{-3}+\gamma^{-3})$.
\end{theorem}
Section~\ref{sec:upper-proof} proves the descent and query estimates. Appendix~\ref{app:caps} verifies that the deterministic iteration and query caps preserve the stated success probability. Each shared-sample difference uses two charged evaluations; the algorithm never needs more than two evaluations of any newly drawn sample.

\subsection{Lower Bound}\label{sec:main-lower}
The lower bound covers the full class of adaptive randomized algorithms in Definition~\ref{def:oracle}, including arbitrary persistent reuse, unrestricted query locations, and outputs that have not been queried.
\begin{theorem}[Lower bound with persistent reuse]\label{thm:lower}
Fix positive bounds $\Delta,L_1,L_2,\sigma_g,\sigma_H,G$, and define
\begin{equation*}
    s_\star=\min\{1,\sigma_g,\sigma_H,L_1,L_2,G\},\qquad
    D_\eps=\left\lceil C\log\left(2+\frac{\Delta}{\eps}
    +\frac{\Delta s_\star^2}{\eps^3}\right)\right\rceil,
\end{equation*}
where $C$ is a sufficiently large numerical constant. For all sufficiently small $\eps,\gamma>0$, there is a dimension threshold $d_\star$ with
\begin{equation*}
    d_\star\le C_d\left(
    1+\frac{\Delta D_\eps^{3/4}}{\eps}
    +\gamma^{-4}\log(2/\gamma)+\gamma^{-3}\right)
\end{equation*}
such that, for $d\ge d_\star$, any algorithm satisfying~\eqref{eq:success} on every instance obeying Assumptions~\ref{ass:gap}--\ref{ass:population-gradient} must have worst-case total query budget at least
\begin{equation}\label{eq:total-lower}
    c\max\left\{
    \frac{\Delta s_\star^2}{\eps^3D_\eps^{1/4}},\ \gamma^{-3}
    \right\}.
\end{equation}
Here $c>0$ and $C_d>0$ may depend on the fixed problem parameters, but not on the tolerances. The same order holds for worst-case expected budgets. It remains necessary under the stronger oracle that returns the sampled value, complete gradient, and complete Hessian together in each unit-cost call.
\end{theorem}
The first-order construction in Section~\ref{sec:firstorder-lb} proves the first term of~\eqref{eq:total-lower}; the deterministic curvature construction in Section~\ref{sec:curvature-lb} proves the second. Section~\ref{sec:lower-proof} combines these separate instance families and proves Theorem~\ref{thm:lower}. Appendix~\ref{app:gradient-test} additionally proves that some admissible instances require $\Omega(\eps^{-2})$ gradient evaluations even if Hessian evaluations are free.

\begin{corollary}[Matching total complexity]\label{cor:total}
Under the oracle and assumptions of Section~\ref{sec:setup}, for fixed positive problem parameters and sufficiently small independent tolerances, the total complexity satisfies
\begin{equation*}
    Q_{\mathrm{tot}}^\star(\eps,\gamma)
    =\widetilde\Theta(\eps^{-3}+\gamma^{-3}),
\end{equation*}
with the growing-dimension convention stated there. The pair in~\eqref{eq:Ng}--\eqref{eq:NH} is achievable; the corollary does not characterize the full $(N_g,N_H)$ Pareto frontier.
\end{corollary}
\begin{proof}
The upper bound follows from Theorem~\ref{thm:upper}. In Theorem~\ref{thm:lower}, $D_\eps$ is logarithmic in the inverse tolerance for fixed positive parameters, and the maximum of two nonnegative terms is at least half their sum. These facts yield the matching lower bound up to logarithmic factors.
\end{proof}
\section{Upper Bound: Derivative Tracking and Descent}\label{sec:upper}
This section proves Theorem~\ref{thm:upper} by combining derivative estimation with descent. Section~\ref{sec:tracking} specifies the $\textsc{TreeEstimate}$ subroutine used in Algorithm~\ref{alg:curvature-first}, proves simultaneous accuracy of its gradient and Hessian estimates, and bounds their sampling cost through predictable drift and martingale motion. Section~\ref{sec:upper-proof} applies these estimates to the optimization trajectory, derives the separate gradient and Hessian budgets, and establishes the total rate $\widetilde{\cO}(\eps^{-3}+\gamma^{-3})$ with success probability at least $2/3$.

\subsection{Binary-Ancestor Derivative Estimation}\label{sec:tracking}
Consider a vector- or self-adjoint-matrix-valued random field $A(x;\xi)$, with mean $A(x)$, satisfying
\begin{equation*}
\norm{A(x;\xi)-A(x)}\le\sigma,\qquad
 \norm{A(x;\xi)-A(y;\xi)}\le L\norm{x-y}.
\end{equation*}
The appropriate norm is understood. The same Lipschitz bound holds for $A$.
Let $y_0,y_1,\ldots$ be adaptively selected locations, with at most $T+1$ locations. Put
\begin{equation*}
m=\lceil\log_2(T+1)\rceil+1,\qquad
 p(k)=k-2^{v_2(k)}\quad(k\ge1),
\end{equation*}
where $v_2(k)$ is the exponent of two in $k$. Repeatedly applying $p$ removes the least significant nonzero bit, so every root-to-node path has at most $m$ edges.

Estimate $A(y_0)$ with a fresh batch. For $k\ge1$, after both endpoints have been determined, use new independent sample handles to form
\begin{equation*}
\widehat D_k=\frac1{b_k}\sum_{i=1}^{b_k}
 [A(y_k;\xi_{ki})-A(y_{p(k)};\xi_{ki})],
 \qquad \widehat A_k=\widehat A_{p(k)}+\widehat D_k.
\end{equation*}
Use independent batches for different edges. An estimate contains only logarithmically many cached edge errors. Algorithm~\ref{alg:tree-estimate} implements one update; its parameters $(\sigma,L,\eta,\delta,T)$ are fixed for each estimator instance. The state consists of the next checkpoint index and all previously cached location--estimate pairs. In Algorithm~\ref{alg:curvature-first}, the two instances use $\delta=1/96$, with error targets $\eps/4$ and $\gamma/4$, respectively. Every evaluation is charged to the corresponding derivative budget, and the caller aborts before any evaluation that would exceed its cap.

\begin{algorithm}[tbp]
\caption{$\textsc{TreeEstimate}(A,x,\eta)$: one binary-ancestor update}\label{alg:tree-estimate}
\begin{algorithmic}[1]
\Require Derivative field $A$; point $x$; error target $\eta$; fixed bounds $\sigma,L$, failure probability $\delta$, and horizon $T$.
\Statex \textbf{Persistent state:} next index $k$, initially $0$; cached pairs $(y_j,\widehat A_j)$ for $j<k$.
\State Set $y_k\gets x$, $m\gets\lceil\log_2(T+1)\rceil+1$, and $\ell\gets\log(C(d+1)(T+1)/\delta)$.
\State Set $\beta\gets C\ell(m+1)^2/\eta^2$, with $C$ as in Lemma~\ref{lem:accuracy}.
\If{$k=0$}
    \State Set $b\gets\max\{1,\lceil\beta\sigma^2\rceil\}$.
\Else
    \State Set $p\gets k-2^{v_2(k)}$.
    \State Set $b\gets\max\{1,\lceil\beta\min\{\sigma^2,L^2\norm{y_k-y_p}^2\}\rceil\}$.
\EndIf
\State Draw $b$ fresh independent sample handles, with samples $\xi_1,\ldots,\xi_b$.
\If{$k=0$}
    \State Compute $\widehat A_0\gets b^{-1}\sum_{i=1}^b A(y_0;\xi_i)$ using $b$ evaluations.
\Else
    \State Compute $\widehat D_k\gets b^{-1}\sum_{i=1}^b[A(y_k;\xi_i)-A(y_p;\xi_i)]$ using $2b$ evaluations.
    \State Set $\widehat A_k\gets\widehat A_p+\widehat D_k$.
\EndIf
\State Cache $(y_k,\widehat A_k)$ and set $k\gets k+1$.
\State \Return $\widehat A_{k-1}$.
\end{algorithmic}
\end{algorithm}

\begin{lemma}[Uniform accuracy]\label{lem:accuracy}
For a target error $\eta>0$ and failure probability $\delta$, it suffices to use
\begin{align*}
 b_0&=\max\left\{1,\left\lceil C\ell(m+1)^2\sigma^2/\eta^2\right\rceil\right\},\\
 b_k&=\max\left\{1,\left\lceil\frac{C\ell(m+1)^2}{\eta^2}
 \min\{\sigma^2,L^2\norm{y_k-y_{p(k)}}^2\}\right\rceil\right\},
\end{align*}
where $\ell=\log(C(d+1)(T+1)/\delta)$ and $C$ is a numerical constant. Then all constructed estimates satisfy $\norm{\widehat A_k-A(y_k)}\le\eta$ simultaneously with probability at least $1-\delta$.
\end{lemma}
The proof conditions on the two endpoints before drawing the batch, applies matrix Hoeffding with self-adjoint dilation for vectors~\cite{tropp2012user}, and sums at most $m$ edge errors. Appendix~\ref{app:technical} contains the full argument.

\subsubsection{The drift--martingale charging lemma}
Suppose a trajectory has increments
\begin{equation*}
x_{t+1}-x_t=a_t+z_t,\qquad
 \bbE[z_t\mid\mathcal F_t]=0,
\end{equation*}
where $a_t$ is $\mathcal F_t$-measurable and $z_t$ has finite second moment. Define
\begin{equation*}
D=\sum_{t<T}\norm{a_t},\qquad V=\sum_{t<T}\norm{z_t}^2.
\end{equation*}
The locations $y_k=x_{\tau_k}$ may be sampled at increasing bounded stopping times. This permits a gradient estimator to be updated only when a curvature test passes.

\begin{lemma}[Dyadic charging]\label{lem:charge}
The binary-ancestor estimator has expected evaluation cost bounded by
\begin{equation*}
C\,\bbE M+
 \frac{C\ell(m+1)^3}{\eta^2}
 \left(\sigma^2+\sigma L\bbE D+L^2\bbE V\right),
\end{equation*}
where $M$ is its number of actual locations, and a same-sample difference is charged twice. This bound also holds for a trajectory stopped at any bounded stopping time.
\end{lemma}
\begin{proof}
For one edge between bounded stopping times $s\le t$, decompose
\begin{equation*}
x_t-x_s=A_{s,t}+Z_{s,t},\quad
 A_{s,t}=\sum_{u=s}^{t-1}a_u,\quad
 Z_{s,t}=\sum_{u=s}^{t-1}z_u.
\end{equation*}
For nonnegative $u,v,c$,
$\min\{c,2u+2v\}\le2\min\{c,u\}+2v$.
Together with $\min\{\sigma^2,L^2r^2\}\le\sigma Lr$, this gives
\begin{align}
 \bbE\min\{\sigma^2,L^2\norm{x_t-x_s}^2\}
 &\le2\sigma L\bbE\sum_{u=s}^{t-1}\norm{a_u}
   +2L^2\bbE\norm{Z_{s,t}}^2\\
 &=2\sigma L\bbE\sum_{u=s}^{t-1}\norm{a_u}
   +2L^2\bbE\sum_{u=s}^{t-1}\norm{z_u}^2.\label{eq:chargeedge}
\end{align}
The last equality is the stopped martingale isometry. It is valid because the endpoints are bounded stopping times, not arbitrary retrospectively selected times.

For any fixed $j$, edges with $2^{v_2(k)}=2^j$ occupy disjoint intervals of the checkpoint index. Their corresponding chronological intervals are disjoint as well. Summing \eqref{eq:chargeedge} over that level charges $D$ and $V$ at most once. There are at most $m$ levels. Summing the batch sizes of Lemma~\ref{lem:accuracy}, including two charges per edge, proves the result. If the number of actual checkpoints is random, extend their stopping times by the terminal stopping time up to the deterministic index cap. The nonnegative increment charges for actual edges are bounded by those of the extended tree. Ceilings are charged only at actual checkpoints, yielding the term $C\bbE M$.
\end{proof}

\paragraph{What the estimate avoids.}
Replacing every squared displacement by the squared sum of step lengths destroys the benefit. For randomly signed steps, it is $\bbE V$, not $\bbE(\sum\norm{z_t})^2$, that enters the cost. The global noise cap is also important: it converts potentially long predictable displacements to a linear drift charge.

\subsection{Proof of the Upper Bound}\label{sec:upper-proof}
We analyze Algorithm~\ref{alg:curvature-first} with the binary-ancestor estimators of Section~\ref{sec:tracking}. The Hessian error target is $\gamma/4$ at every iterate, and the gradient error target is $\eps/4$ on the subsequence that passes the Hessian test. A new batch of independent sample handles is drawn only after the two endpoints of its estimator edge have been determined.

\subsubsection{Progress, predictable drift, and martingale energy}
Suppose estimates used for a step are accurate. At a gradient step,
\begin{equation*}
\left\langle\nabla f(x_t),\frac{g_t}{\norm{g_t}}\right\rangle
 \ge\norm{g_t}-\eps/4>\eps/2.
\end{equation*}
Population $L_1$-smoothness therefore gives
\begin{equation}
f(x_{t+1})-f(x_t)
 \le-\frac{\eps^2}{8L_1}+\frac{\eps^2}{32L_1}
 =-d_g,\qquad d_g:=\frac{3\eps^2}{32L_1}.\label{eq:dg}
\end{equation}
At a curvature step, $v_t^\top\nabla^2f(x_t)v_t<-\gamma/2$. The two signs cancel the linear term in the Hessian-Lipschitz Taylor inequality, giving
\begin{align*}
 \bbE[f(x_{t+1})-f(x_t)\mid\mathcal F_t]
 &\le-\frac\gamma4\left(\frac\gamma{L_2}\right)^2
       +\frac{L_2}{6}\left(\frac\gamma{L_2}\right)^3\\
 &=-d_H,\qquad d_H:=\frac{\gamma^3}{12L_2^2}.
\end{align*}
Here $\mathcal F_t$ includes the chosen eigenvector and all oracle data, but not the new sign.

For a rigorous failure-event analysis, stop an analytical shadow trajectory before executing the first step with an inaccurate estimator, and freeze it thereafter. This stopping rule is for the proof, not an implementable oracle. Also stop upon successful return or at the deterministic iteration cap. Denote its executed gradient and curvature counts by $K_g,K_H$. Telescoping conditional expected descent and using $f\ge\inf f$ yields
\begin{equation*}
d_g\bbE K_g+d_H\bbE K_H\le\Delta.
\end{equation*}
Write gradient steps as the predictable increments $a_t$ and signed curvature steps as $z_t$. Then
\begin{align}
 \bbE D&\le\frac{\eps}{4L_1}\frac{\Delta}{d_g}
       =\frac{8\Delta}{3\eps},\label{eq:D}\\
 \bbE V&\le\frac{\gamma^2}{L_2^2}\frac{\Delta}{d_H}
       =\frac{12\Delta}{\gamma}.\label{eq:V}
\end{align}
These identities remain valid after the shadow stopping rule because the decision to execute a signed step is made before its independent sign is drawn.

\subsubsection{Separate query counts}
The shadow has at most $K_g+1$ actual gradient checkpoints and at most $K_g+K_H+1$ Hessian checkpoints. A final inaccurate test or successful stopping test is included in the extra one. For the gradient tree apply Lemma~\ref{lem:charge} with $(\sigma,L,\eta)=(\sigma_g,L_1,\eps/4)$ at its bounded stopping-time subsequence. For the Hessian tree use $(\sigma_H,L_2,\gamma/4)$ at every state. The initial estimation costs are proportional to $\sigma_g^2/\eps^2$ and $\sigma_H^2/\gamma^2$. The drift bound~\eqref{eq:D} contributes $\sigma_gL_1\Delta/\eps^3$ and $\sigma_HL_2\Delta/(\eps\gamma^2)$, while the martingale bound~\eqref{eq:V} contributes $L_1^2\Delta/(\eps^2\gamma)$ and $L_2^2\Delta/\gamma^3$. Adding the checkpoint counts yields the budgets $B_g$ and $B_H$ defined in~\eqref{eq:B_g}--\eqref{eq:B_H}, up to logarithmic factors. The two curvature terms combined in $B_H$ account for the checkpoint ceiling and martingale contribution, respectively.

\subsubsection{Fixed budgets and the total rate}
The estimates above concern the shadow process, which stops before its first inaccurate step. Appendix~\ref{app:caps} converts these expectations into an implementable algorithm with deterministic iteration and derivative caps. A conditional union bound controls estimation failure; Markov bounds control timeout and the two resource overflows. The resulting success probability exceeds $2/3$.
At a successful return,
\begin{equation*}
\norm{\nabla f(x_t)}\le\norm{g_t}+\eps/4\le\eps,
 \qquad
 \lammin(\nabla^2f(x_t))\ge\lammin(H_t)-\gamma/4\ge-\gamma.
\end{equation*}
For fixed positive parameters and small tolerances, \eqref{eq:B_g}--\eqref{eq:B_H} simplify to \eqref{eq:Ng}--\eqref{eq:NH}. Finally, weighted arithmetic--geometric mean gives
\begin{equation*}
\eps^{-2}\gamma^{-1}\le\tfrac23\eps^{-3}+\tfrac13\gamma^{-3},\qquad
 \eps^{-1}\gamma^{-2}\le\tfrac13\eps^{-3}+\tfrac23\gamma^{-3}.
\end{equation*}
This completes the proof of Theorem~\ref{thm:upper}.

\section{Lower Bound: Hard Instances under Sample Reuse}\label{sec:lower}
We prove Theorem~\ref{thm:lower} through two separate families of admissible instances. Section~\ref{sec:curvature-lb} constructs deterministic losses requiring $\Omega(\gamma^{-3})$ queries to return an $(\eps,\gamma)$-SOSP, even with unrestricted adaptive queries and complete-Hessian responses; the tolerance dependence arises from negative curvature along a hidden chain. Section~\ref{sec:firstorder-lb} constructs stochastic losses requiring $\Omega(\Delta s_\star^2/(\eps^3D_\eps^{1/4}))$ total queries to satisfy the gradient condition, even with arbitrary persistent sample reuse. This yields $\widetilde\Omega(\eps^{-3})$ for fixed positive scale parameters. Both constructions permit joint value, gradient, and complete-Hessian responses at unit cost. Section~\ref{sec:lower-proof} embeds the two families into a common dimension and combines their bounds to complete the proof, including worst-case expected budgets. The first-order bound concerns the tolerance $\eps$ and does not restrict the algorithm to gradient-only access.

\subsection{A Curvature Lower Bound for Unrestricted Queries}\label{sec:curvature-lb}
A lower bound proved only for coordinate-restricted queries would not suffice. Also, simply making a chain compactly supported creates easy stationary outputs far away. The construction below deals with both issues and respects the global population-gradient bound.

\subsubsection{A bounded smooth curvature chain}
Define
\begin{equation*}
s(u)=\begin{cases}0,&u\le\tfrac12,\\
             \exp(-1/(u-\tfrac12)),&u>\tfrac12,
             \end{cases}
 \qquad a(u)=s(u)+s(-u),\qquad
 b(u)=4(e^{-u^2/4}-1).
\end{equation*}
These are $C^\infty$ with bounded derivatives of every fixed order. We have $0\le a<1$, $-4\le b\le0$, and $a$, with all its derivatives, vanishes on $[-1/2,1/2]$. Furthermore,
\begin{equation*}
a(1)=e^{-2},\quad a''(u)\ge0\ (|u|\le1),\quad
 b''(u)=(u^2-2)e^{-u^2/4}\le-e^{-1/4}\ (|u|\le1).
\end{equation*}
For example, if $t=u-1/2\in(0,1/2]$, then
$s''(u)=e^{-1/t}(1-2t)/t^4\ge0$, proving the gate-convexity assertion.

For chain length $n$, let
\begin{equation*}
q_n(y)=a(1)b(y_1)+\sum_{i=2}^n a(y_{i-1})b(y_i).
\end{equation*}
There are numerical constants $C_1,C_2,C_3$ independent of $n$ such that
\begin{align}
 -4n\le q_n\le0,
 &\qquad q_n(0)=0,\label{eq:qgap}\\
 \norm{\nabla q_n}\le C_1\sqrt n,
 &\qquad \norm{\nabla^2q_n}\le C_2,\qquad
 \operatorname{Lip}(\nabla^2q_n)\le C_3.\label{eq:qreg}
\end{align}
Indeed, a gradient coordinate contains at most two bounded terms. The Hessian is tridiagonal with uniformly bounded row sums. Every third derivative is a sum of bounded coefficients times products of three vector coordinates at neighboring indices; Holder's inequality, followed by $\norm{v}_3\le\norm{v}$, bounds its trilinear operator norm independently of $n$.

\begin{lemma}[Curvature and robust sequential revelation]\label{lem:chain}
Let $c_0=e^{-9/4}$.
If any coordinate of $y$ has magnitude at most one, then
\begin{equation*}
\lammin(\nabla^2q_n(y))\le-c_0.
\end{equation*}
If $|y_i|\le1/2$ for every $i\ge k$, the value, gradient, and complete Hessian of $q_n$ at $y$ depend only on its first $k$ coordinates. Equivalently, its order-two jet agrees there with the truncated chain $q_k$.
\end{lemma}
\begin{proof}
Let $j$ be the first index with $|y_j|\le1$. If $j=1$, interpret its predecessor coefficient as $a(1)$; otherwise $|y_{j-1}|>1$, so that $a(y_{j-1})\ge a(1)$. The diagonal Hessian entry is
\begin{equation*}
\partial_{jj}q_n(y)=a(y_{j-1})b''(y_j)+a''(y_j)b(y_{j+1})\le-c_0,
\end{equation*}
where the second term is omitted at $j=n$. This proves the curvature claim. For the jet claim, each omitted term $a(y_{i-1})b(y_i)$, $i>k$, has a gate whose value and first two derivatives vanish. Thus the term and its complete order-two jet vanish. This remains true at the gate endpoints because the gate is flat there.
\end{proof}

\subsubsection{Scaling and bounded-gradient radial globalization}
Let $U\in\bbR^{d\times n}$ have orthonormal columns. Choose fixed constants
\begin{equation*}
\kappa=8/c_0,\qquad K\ge4\kappa C_3/L_2,
 \qquad h=K\gamma,\qquad \lambda=\kappa\gamma h^2.
\end{equation*}
For sufficiently small $\gamma$, set
\begin{equation*}
n=\left\lfloor\frac{\Delta}{8\lambda}\right\rfloor,
 \qquad \phi_U(x)=\lambda q_n(U^\top x/h).
\end{equation*}
Then $n=\Theta(\gamma^{-3})$ and
\begin{equation*}
-\Delta/2\le\phi_U\le0,\quad
 \norm{\nabla\phi_U}\le B_1\sqrt\gamma,\quad
 \norm{\nabla^2\phi_U}\le B_2\gamma,\quad
 \operatorname{Lip}(\nabla^2\phi_U)\le L_2/4,
\end{equation*}
for constants $B_1,B_2$ depending only on the fixed model parameters. At a point with an incomplete chain coordinate, the direction of that coordinate has curvature at most $-8\gamma$.

Let $b_0=G/4$. Choose fixed smooth functions $S,\chi:[0,\infty)\to[0,1]$ such that $S$ is nondecreasing, equals zero on $[0,1/2]$, and equals one on $[1,\infty)$; and $\chi$ is nonincreasing, equals one on $[0,1]$, and equals zero on $[2,\infty)$. Put
\begin{equation*}
R=K_R/\gamma,\qquad K_R\ge\max\{1,2b_0\},\qquad
 W_R(r)=b_0R\int_0^{r/R}S(u)\,\rmd u,
\end{equation*}
and define
\begin{equation}
f_U(x)=\chi(\norm{x}/R)\phi_U(x)+W_R(\norm{x}).\label{eq:global}
\end{equation}
The radial functions are constant near the origin, so they are smooth as functions of $x$ there. All cutoff transitions are smooth and flat at their plateau endpoints.

\begin{lemma}[Admissibility and exclusion of exterior outputs]\label{lem:global}
For all sufficiently small $\gamma$, the deterministic loss $F(x;\xi)=f_U(x)$ satisfies every assumption of the model, with $x_0=0$, uniformly over $U$. For all sufficiently small $\eps$, an $(\eps,\gamma)$-SOSP must lie in $\norm{x}<R$. Inside that ball, an $(\eps,\gamma)$-SOSP must have $|(U^\top x/h)_i|>1$ for every $i$.
\end{lemma}
The radial derivative outside the core is bounded below by $G/8$ for sufficiently small tolerances. Inside the transition annulus, projecting an incomplete hidden direction onto the tangent space preserves its negative curvature, while the wall contributes only $\cO(1/R)$. Appendix~\ref{app:technical} verifies all constants and derivatives.

\subsubsection{Arbitrary adaptive queries and full-Hessian information}
Choose $U$ uniformly at random from orthonormal $n$-frames. Fix an arbitrary randomized algorithm, with internal coins independent of $U$, and a call cap $N<n$, for example $N=\lfloor n/2\rfloor$. A gradient and a Hessian at the same point may be charged as two calls; the argument below is stronger and permits an entire value--gradient--Hessian jet in a single call.

Construct an ideal transcript sequentially, exposing columns of $U$ only to the analysis, not as additional data to the algorithm. At call $t$, the query is measurable with respect to the algorithm's coins and the first $t-1$ exposed columns, because every preceding ideal response is a function of those variables. Expose $u_t$ in the analysis. If the query norm is greater than $2R$, return only the radial wall jet, independent of $U$. Otherwise return the jet of the cutoff version of $\lambda q_t(U_{1:t}^\top x/h)$ plus the same wall. The algorithm receives just this ideal oracle response; no hidden column is supplied separately.

Conditional on the preceding columns, each unrevealed $u_i$ is uniformly distributed on the unit sphere in an orthogonal complement of dimension at least $d-N$. For any already chosen query of norm at most $2R$,
\begin{equation*}
\bbP(|u_i^\top x|>h/2\mid\text{past columns and coins})
 \le 2\exp(-c(d-N)h^2/R^2).
\end{equation*}
The same bound applies to the output, if its norm is at most $2R$, conditionally on the first $N$ columns. Take
\begin{equation*}
d\ge n+N+C(R/h)^2\log(Cn(N+1)).
\end{equation*}
A union bound over all calls, unrevealed columns, and the final output makes the probability of any large unrevealed projection at most $1/12$. Here $d=\cO(\gamma^{-4}\log(1/\gamma)+\gamma^{-3})$, with constants depending only on fixed parameters.

On the complementary event, Lemma~\ref{lem:chain} says that every ideal jet is exactly the full jet of the single fixed function $f_U$. Consequently the ideal and true transcripts coincide inductively. The final output either has norm at least $R$, in which case its gradient is too large, or has norm less than $R$, in which case its $(N+1)$st chain coordinate has magnitude at most $1/2$ and its Hessian has an eigenvalue below $-\gamma$. Thus the average success probability of any $N$-call algorithm on this distribution is at most $1/12$. Some fixed $U$ is therefore hard for that algorithm. This proves $\Omega(n)=\Omega(\gamma^{-3})$ total complexity without restricting query norms, algorithmic spans, or the returned Hessian matrix.

For expected budgets, suppose a uniformly successful algorithm had worst-case expected total cost less than $N/6$. Truncating at $N$ calls reduces its success probability by at most $1/6$, leaving at least $1/2$ success on every instance. This contradicts the fixed-cap lower bound just proved. Hence the expected lower bound has the same order.

\subsection{A First-Order Lower Bound with Persistent Handles}
\label{sec:firstorder-lb}
The purpose of this section is to close the remaining $\gamma>\eps$ regime. A standard fresh-sample lower bound is insufficient: repeated evaluations of one sampled loss can reveal its entire local geometry. We therefore construct each sample loss from a finite coefficient table and prove the lower bound even for a stronger stateful oracle that reveals the relevant table entry on its first use. Reusing that handle can then provide no additional information about the current hidden block.

\subsubsection{Third-order reverse stick-breaking localization}
The conservative scalar embedding uses the reverse stick-breaking weights
\begin{equation*}
P_p(x)=\prod_{k=p+1}^{T}(1-b_k(x_k)),\qquad
 w_0=P_0,\qquad w_i=b_iP_i\quad(1\le i\le T),
\end{equation*}
where the blocks $x_i$ are disjoint. The weights are nonnegative and sum to one. We need one derivative order beyond the length-independent value, gradient, and Hessian bounds used in the fresh-oracle construction.

\begin{lemma}[Third-order reverse stick localization]\label{lem:thirdstick}
Let $\beta_i=b_i(1-b_i)$. Suppose the $C^3$ gates satisfy
\begin{equation*}
\norm{Db_i}^2\le \ell_1^2\beta_i,\qquad
 \norm{D^2b_i}^2\le \ell_2^2\beta_i,\qquad
 \norm{D^3b_i}\le \ell_3
\end{equation*}
uniformly. Let $V_0,\ldots,V_T$ be scalar functions on the full product space with
\begin{equation*}
\norm{D^rV_i}\le v_r,
 \qquad r=0,1,2,3,
\end{equation*}
where $D^0V_i=V_i$. Then $N=\sum_iw_iV_i$ obeys, with universal constants independent of $T$,
\begin{align}
 |N|&\le v_0,\label{eq:stick0}\\
 \norm{DN}&\le v_1+C v_0\ell_1,\label{eq:stick1}\\
 \norm{D^2N}&\le v_2+C\{v_1\ell_1+v_0(\ell_2+\ell_1^2)\},\label{eq:stick2}\\
 \norm{D^3N}&\le v_3+C\{v_2\ell_1+v_1(\ell_2+\ell_1^2)
 +v_0(\ell_3+\ell_1\ell_2+\ell_1^3)\}.\label{eq:stick3}
\end{align}
\end{lemma}
\paragraph{Proof mechanism.}
Introduce independent Bernoulli variables with success probabilities $b_i$ and let $L$ be the last successful index. Then $N=\bbE_BV_L$. The first, second, and third derivatives of the Bernoulli product mass decompose into score monomials of degrees at most three. Monomials indexed by different subsets are orthogonal. The weighted gate hypotheses control their first two moments without a factor depending on $T$; the degree-one third-order term is controlled in $L^1$. The product rule then yields the displayed bounds. Appendix~\ref{app:technical} proves these score estimates and handles gates attaining zero or one.

\subsubsection{A slow-coordinate conservative chain}
Use the flat gate
\begin{equation*}
\rho(t)=
 \begin{cases}
 0,&t\le0,\\
 35t^4-84t^5+70t^6-20t^7,&0<t<1,\\
 1,&t\ge1.
 \end{cases}
\end{equation*}
Besides bounded derivatives through order three, it satisfies
\begin{equation}\label{eq:rho-weighted}
 (\rho')^2\le C\rho(1-\rho),\qquad
 (\rho'')^2\le C\rho(1-\rho).
\end{equation}
Indeed, on $[0,1]$ one has $\rho(1-\rho)\ge t^4(1-t)^4$, while
$\rho'=140t^3(1-t)^3$ and
$\rho''=420t^2(1-t)^2(1-2t)$.

Let $D$ be a block size, let
\begin{equation*}
\alpha=D^{1/4},
\end{equation*}
and draw independent hidden vectors $\theta_j\in\{-1,1\}^D$. For $x_j\in\bbR^D$, put
\begin{equation*}
q_j(x_j)=\frac2D\sum_{\ell=1}^D\theta_{j\ell}\tanh(\alpha x_{j\ell}),
 \qquad
 b_j(x_j)=\rho(2(q_j(x_j)-1/4)).
\end{equation*}
For the scalar gate $b(q)=\rho(2(q-1/4))$, define
\begin{equation*}
h(q)=-2(3+q_1)-4\sum_{j=1}^{T-1}b(q_j)(3+q_{j+1}),
\end{equation*}
and the population objective
\begin{equation}\label{eq:slow-chain}
 \Phi_\Theta(x)=A\left[
 h(q(x))+\frac8D\sum_{j=1}^{T}\sum_{\ell=1}^{D}
 \log\cosh(\alpha x_{j\ell})
 \right].
\end{equation}

\begin{lemma}[Slow-chain geometry]\label{lem:slowgeom}
There are universal $C,c_0>0$ such that
\begin{align}
 \Phi_\Theta(0)-\inf\Phi_\Theta&\le CAT,\label{eq:slow-gap}\\
 \sup_x\norm{\nabla\Phi_\Theta(x)}&\le CA\alpha\sqrt{T/D},\label{eq:slow-grad-global}\\
 \sup_x\norm{\nabla^2\Phi_\Theta(x)}&\le CA\alpha^2/D,\label{eq:slow-hess}\\
 \operatorname{Lip}(\nabla^2\Phi_\Theta)&\le CA\alpha^3/D.\label{eq:slow-third}
\end{align}
If $q_j(x_j)\le1/2$ for at least one block, then
\begin{equation}\label{eq:slow-stationarity}
 \norm{\nabla\Phi_\Theta(x)}\ge c_0A\alpha/\sqrt D.
\end{equation}
\end{lemma}
The derivative bounds follow from interactions between adjacent blocks. For stationarity, take the first block with $q_j\le1/2$; the incoming chain coefficient has magnitude at least two, while the regularizer forces the block gradient away from zero. The scalar factorization and the complete derivative calculation appear in Appendix~\ref{app:technical}.

Choose
\begin{equation}\label{eq:slow-scales}
 A=\frac{4\eps\sqrt D}{c_0\alpha}
 =\frac4{c_0}\eps D^{1/4},
 \qquad
 T=\left\lfloor c_T\frac{\Delta}{A}\right\rfloor.
\end{equation}
with $c_T$ small enough that the initial gap is at most $\Delta/2$. Then every $\eps$-stationary point has $q_j>1/2$ for all $j$. Moreover,
\begin{equation*}
\operatorname{Lip}(\nabla\Phi_\Theta)=\cO(\eps D^{-1/4}),
 \qquad
 \operatorname{Lip}(\nabla^2\Phi_\Theta)=\cO(\eps),
\end{equation*}
and \eqref{eq:slow-grad-global}, together with $T\le C\Delta/A$, gives
\begin{equation*}
\sup_x\norm{\nabla\Phi_\Theta(x)}
 \le C\sqrt{\Delta\eps}\,D^{-1/8}.
\end{equation*}
Thus all population regularity requirements hold for sufficiently small $\eps$.

\subsubsection{One scalar sampled loss with a fixed sample-wise Hessian Lipschitz constant}
For one sample handle and each block, draw independent coefficients
\begin{equation*}
Z_{j\ell}\in\{-R,R\},\qquad
 \Pr(Z_{j\ell}=R)=\frac{1+\theta_{j\ell}/R}{2},
 \qquad R\ge2.
\end{equation*}
Define
\begin{equation*}
\widehat q_j(x_j;Z_j)=\frac2D\sum_\ell Z_{j\ell}\tanh(\alpha x_{j\ell}),
 \qquad
 n_j=-2(\widehat q_j-q_j).
\end{equation*}
With $P_p=\prod_{k=p+1}^{T}(1-b_k)$, set
\begin{equation}\label{eq:slow-sample}
 \eta_\Theta(x;Z)=P_0n_1+2\sum_{p=1}^{T-1}b_pP_pn_{p+1},
 \qquad
 F_\Theta(x;Z)=\Phi_\Theta(x)+A\eta_\Theta(x;Z).
\end{equation}
The coefficient means give $\bbE_ZF_\Theta=\Phi_\Theta$ and the same identity for the first two derivatives.

For every block and realization,
\begin{align}
 |n_j|&\le C(R+1),\nonumber\\
 \norm{Dn_j}&\le C(R+1)\alpha/\sqrt D,\nonumber\\
 \norm{D^2n_j}&\le C(R+1)\alpha^2/D,\nonumber\\
 \norm{D^3n_j}&\le C(R+1)\alpha^3/D.\label{eq:n-derivatives}
\end{align}
The gate estimates, using \eqref{eq:rho-weighted}, are
\begin{align}
 \norm{Db_j}^2&\le C\frac{\alpha^2}{D}b_j(1-b_j),\nonumber\\
 \norm{D^2b_j}^2&\le C\frac{\alpha^4}{D^2}b_j(1-b_j),\nonumber\\
 \norm{D^3b_j}&\le C\frac{\alpha^3}{D}.\label{eq:gate-three}
\end{align}
Applying Lemma~\ref{lem:thirdstick} to \eqref{eq:slow-sample} therefore yields, independently of $T$,
\begin{align}
 |\eta_\Theta|&\le C(R+1),\nonumber\\
 \norm{\nabla\eta_\Theta}&\le C(R+1)\alpha/\sqrt D,\nonumber\\
 \norm{\nabla^2\eta_\Theta}&\le C(R+1)\alpha^2/D,\nonumber\\
 \norm{D^3\eta_\Theta}&\le C(R+1)\alpha^3/D.\label{eq:eta-three}
\end{align}

Put
\begin{equation}\label{eq:R-slow}
 s_\star=\min\{1,\sigma_g,\sigma_H,L_1,L_2,G\},
 \qquad R=c_Rs_\star/\eps.
\end{equation}
where $c_R$ is a sufficiently small universal constant. Since
$A\alpha/\sqrt D=4\eps/c_0$ and $\alpha^2/\sqrt D=1$, equations
\eqref{eq:eta-three} imply
\begin{align*}
 \sup_{x,Z}\norm{\nabla F_\Theta-\nabla\Phi_\Theta}
 &\le C\eps(R+1)\le\sigma_g,\\
 \sup_{x,Z}\norm{\nabla^2F_\Theta-\nabla^2\Phi_\Theta}
 &\le C\eps(R+1)D^{-1/4}\le\sigma_H,\\
 \sup_{x,Z}\norm{D^3F_\Theta(x;Z)}
 &\le C\eps+C\eps(R+1)\le L_2,\\
 \sup_{x,Z}\norm{\nabla^2F_\Theta(x;Z)}
 &\le C\eps D^{-1/4}+C\eps(R+1)D^{-1/4}\le L_1
\end{align*}
for all sufficiently small $\eps$. Thus the sampled losses satisfy the required \emph{uniform} sample-wise gradient and Hessian smoothness constants. This is the point of the scale $\alpha=D^{1/4}$; the original $\alpha=\sqrt D$ chain would leave a sample third derivative growing polynomially with $D$.

\subsubsection{Persistent handles and stagewise information}
A handle $r$ contains an independent table
$Z^{(r)}=(Z^{(r)}_{j\ell})_{j,\ell}$ drawn once from the distribution above and retained forever. Grant the algorithm the following stronger charged call: at a chosen point and handle, return the value, full gradient, and full Hessian of the sampled scalar loss, together with every coefficient array of that handle that has become relevant at the current frontier. This dominates either a gradient or a Hessian evaluation in the original model.

Let a block be active when $q_j(x_j)>1/4$. If $x_j$ is independent of a fresh hidden sign vector $\theta_j$, then, uniformly over arbitrarily large $x_j$,
\begin{equation}\label{eq:slow-tail}
 \Pr(q_j(x_j)>1/4)
 \le2e^{-D/128}.
\end{equation}
This is Hoeffding's inequality applied to
$\sum_\ell\theta_{j\ell}\tanh(\alpha x_{j\ell})$.

At a point whose largest active index is $p<T$, set every future gate equal to the constant zero and truncate the products accordingly. Flatness of the zero gate through order two makes the resulting value--gradient--Hessian jet identical to that of the full loss. In the truncated scalar loss, the last population link and its centered perturbation combine as
\begin{equation*}
2b_pW_{\theta_{p+1}}+2b_p(\widehat W_{p+1}-W_{\theta_{p+1}})
 =2b_p\widehat W_{p+1};
\end{equation*}
for $p=0$, the same identity is $W_{\theta_1}+n_1=\widehat W_1$. Hence the complete order-two jet depends only on $\theta_{\le p}$, earlier coefficient arrays, and the next coefficient array $Z^{(r)}_{p+1}$; there is no remaining direct dependence on $\theta_{p+1}$. Replacing that array by $D$ fair signs defines a reference response independent of $\theta_{p+1}$. On its first revelation for a fixed handle,
\begin{equation}\label{eq:handle-kl}
 \operatorname{KL}(\mathcal L(Z^{(r)}_{p+1}\mid\theta_{p+1})
 \|\mathcal L(Z^{0,(r)}_{p+1}))
 \le CD/R^2=: \kappa.
\end{equation}
After that coefficient array has been revealed, every later response using the same handle and frontier is a deterministic function of the transcript and contributes zero additional conditional relative entropy. Thus reuse cannot multiply information from one handle.

\begin{lemma}[Sequential cost with persistent handles]\label{lem:persistent-info}
Suppose the hidden blocks are independent, the far-activation bound is $ce^{-c'D}$, every successful output activates all $T$ blocks, and the first revelation of one handle--next-block array has conditional KL at most $\kappa$ from a reference array independent of that next block. There are universal constants $c,C>0$ such that, if
\begin{equation*}
D\ge C\log(16(N+1)T),\qquad
 n_0=\lfloor cD/\kappa\rfloor\ge1,
\end{equation*}
then every randomized adaptive algorithm making $N\le Tn_0/2$ charged calls has success probability less than $2/3$ on some deterministic hidden instance, even with arbitrary handle reuse.
\end{lemma}
\begin{proof}
Pad a shorter algorithm to a deterministic horizon $N$. Let $x_t$ and $r_t$ be its point and handle at response time $t$, as determined by its private seed and the preceding comparison transcript. Starting from $\nu_0=0$, define
\begin{equation*}
\rho_t=\max(\{j\le\min\{T,\nu_{t-1}+1\}:A_j(x_{t,j})\}\cup\{0\}),
 \qquad
 \nu_t=\max\{\nu_{t-1},\rho_t\},
\end{equation*}
and return the extended rich response of index $\rho_t$. Apply the same update to the final output, without drawing a response. Thus historical progress rises by at most one at a time.

Up to the instant before progress first exceeds a prefix $k$, this transcript is measurable with respect to $\theta_{\le k+1}$, the algorithmic seed, and the sampled coefficient tables already requested from those blocks. Every $\theta_j$ with $j>k+1$ therefore retains its independent prior. A conditional union bound using the far-activation hypothesis gives
\begin{equation}\label{eq:persistent-far}
 \Pr\!\left(\exists t\le N+1,\ \exists j>\nu_{t-1}+1:
 A_j(x_{t,j})\right)
 \le C(N+1)Te^{-c'D}.
\end{equation}
Outside this event, the comparison and actual stateful-oracle transcripts coincide.

For stage $i$, let $\tau_i$ be the first time $\nu_t\ge i$, and let
\begin{equation*}
L_i=\sum_{t=1}^{N}\mathbf 1\{t<\tau_i,\ \rho_t=i-1\}.
\end{equation*}
Each response has only one historical index, hence pathwise $\sum_iL_i\le N$. At the entry to stage $i$, stop immediately before the response that first reaches index $i-1$. The stopped history and the entry query depend only on $\theta_{<i}$ and independent randomization, so conditioning on them does not alter the prior of $\theta_i$.

Continue the algorithm while computing response indices from blocks at most $i-1$, and stop after $n_0$ responses of index $i-1$. Define a reference continuation by replacing, on the first index-$(i-1)$ use of each handle $r$, its array $Z_i^{(r)}$ by a fair array independent of $\theta_i$. If that handle--block array appeared earlier in the continuation, the next response is already a deterministic function of the common transcript and its conditional KL contribution is zero. Otherwise \eqref{eq:handle-kl} bounds the contribution by $\kappa$. The conditional chain rule therefore gives
\begin{equation}\label{eq:persistent-stage-kl}
 \operatorname{KL}(P_i\|Q_i)\le n_0\kappa\le cD.
\end{equation}
Under $Q_i$, the complete adaptive point sequence, including points chosen from previously revealed fair arrays of block $i$, is independent of $\theta_i$. Thus the chance that a continuation query or the output activates block $i$ before the quota is at most
\begin{equation*}
q_D=C(N+1)e^{-c'D}.
\end{equation*}
If $p_i$ denotes that probability in the true continuation, data processing and
$\operatorname{kl}(p\|q)\ge p\log(1/q)-\log2$ imply
\begin{equation*}
n_0\kappa\ge p_i\log(1/q_D)-\log2.
\end{equation*}
The stipulated lower bound on $D$ and a sufficiently small choice of the constant in $n_0$ yield $p_i\le1/32$, uniformly over all entry histories.

If all $T$ stages are completed and $N\le Tn_0/2$, then at least $T/2$ stages have $L_i<n_0$. Markov's inequality therefore gives completion probability at most
\begin{equation*}
\frac2T\sum_{i=1}^T\Pr(\tau_i\le N+1,\ L_i<n_0)\le\frac1{16}.
\end{equation*}
Increasing the logarithmic constant makes \eqref{eq:persistent-far} at most $1/16$. Hence the average success probability over the hidden-block prior is below $1/8$, and some deterministic hidden instance has success probability below $2/3$. The argument charges at most one biased-array revelation per response and never treats evaluations of a reused handle as independent.
\end{proof}

\begin{theorem}[Reusable-handle first-order endpoint]\label{thm:fo-lower}
Let
\begin{equation*}
D=\left\lceil C\log\!\left(2+\frac{\Delta}{\eps}
 +\frac{\Delta s_\star^2}{\eps^3}\right)\right\rceil.
\end{equation*}
For all sufficiently small $\eps$, in dimension
\begin{equation*}
d=DT=\cO\!\left(\frac{\Delta D^{3/4}}{\eps}\right),
\end{equation*}
there is an admissible sample-reuse instance for which every randomized adaptive algorithm needs
\begin{equation*}
\boxed{
 \Omega\!\left(\frac{\Delta s_\star^2}
 {\eps^3D^{1/4}}\right)
 }
\end{equation*}
total derivative evaluations to return an $\eps$-stationary point with probability at least $2/3$. The result remains true if every charged call returns the sampled value, full gradient, and full Hessian simultaneously.
\end{theorem}
\begin{proof}
Use \eqref{eq:slow-scales}, \eqref{eq:R-slow}, and the scalar loss \eqref{eq:slow-sample}. The construction above verifies every model assumption. Every successful output activates every block by \eqref{eq:slow-stationarity}. Equations \eqref{eq:handle-kl} and Lemma~\ref{lem:persistent-info} give
\begin{equation*}
n_0\ge cR^2\ge c s_\star^2/\eps^2.
\end{equation*}
Also
\begin{equation*}
T\ge c\Delta/A
 \ge c\frac{\Delta}{\eps D^{1/4}}.
\end{equation*}
Thus $Tn_0/2$ has the displayed order. The chosen logarithmic $D$ dominates
$\log((N+1)T)$ throughout the claimed budget range, so the horizon condition in Lemma~\ref{lem:persistent-info} holds. Since every $(\eps,\gamma)$-SOSP is in particular $\eps$-stationary, the same lower bound applies for every $\gamma>0$.

For worst-case expected cost, truncate a purported uniformly successful algorithm at a sufficiently large constant multiple of its expected budget. The truncation loses only a fixed amount of success probability by Markov's inequality, while the fixed-cap proof above can be run with average success bounded below that remaining constant. This changes only numerical constants.
\end{proof}

\subsection{Proof of the Lower Bound}\label{sec:lower-proof}
We now combine the two constructions to prove the main lower bound. The first-order family controls the cost of the gradient condition, and the curvature family controls the cost of the Hessian condition; each family satisfies all assumptions of the model.
\begin{proof}[Proof of Theorem~\ref{thm:lower}]
Let $D_\eps$ and $s_\star$ be as in the theorem. Theorem~\ref{thm:fo-lower}, with $D=D_\eps$, establishes a lower bound
\begin{equation*}
    c_\eps\frac{\Delta s_\star^2}{\eps^3D_\eps^{1/4}}
\end{equation*}
for achieving $\norm{\nabla f(\hat x)}\le\eps$ with probability at least $2/3$. Its dimension $d_\eps$ is at most $C_\eps(1+\Delta D_\eps^{3/4}/\eps)$. Every output satisfying both conditions in~\eqref{eq:success} satisfies this gradient condition, so the same lower bound applies for every $\gamma>0$. Section~\ref{sec:curvature-lb} establishes the independent lower bound $c_\gamma\gamma^{-3}$ for~\eqref{eq:success} in a dimension $d_\gamma$ at most $C_\gamma(1+\gamma^{-4}\log(2/\gamma)+\gamma^{-3})$, for sufficiently small $\eps$ and $\gamma$.

Choose $d_\star=\max\{d_\eps,d_\gamma\}$. For any $d\ge d_\star$, embed either construction by setting $\overline F(x;\xi)=F(Px;\xi)$, where $P$ projects onto its original coordinates. The extra gradient entries and Hessian rows and columns are zero. Hence the initial gap, noise bounds, and all smoothness bounds are preserved, and a successful output for $\overline f$ projects to a successful output for $f$. The additional zero eigenvalues do not alter the condition $\lammin(\nabla^2 f)\ge-\gamma$. Responses in the extra coordinates reveal no information, so the original oracle can simulate the enlarged instance under arbitrary adaptive queries and persistent reuse. The displayed bounds on $d_\eps$ and $d_\gamma$ imply the asserted bound on $d_\star$ after enlarging $C_d$.

An algorithm that succeeds on every admissible instance must therefore have worst-case total budget $Q$ satisfying both
\begin{equation*}
    Q\ge c_\eps\frac{\Delta s_\star^2}{\eps^3D_\eps^{1/4}},
    \qquad Q\ge c_\gamma\gamma^{-3}.
\end{equation*}
Taking $c=\min\{c_\eps,c_\gamma\}$ proves~\eqref{eq:total-lower}. Both component bounds hold when a unit-cost call returns the sampled value, gradient, and complete Hessian jointly. In the curvature construction the sample loss is deterministic, so revisiting a handle supplies the same deterministic responses; the first-order construction explicitly permits arbitrary persistent reuse. Thus the combined lower bound holds for the stronger joint oracle as well. The expected-budget arguments at the ends of Section~\ref{sec:curvature-lb} and Section~\ref{sec:firstorder-lb} establish the same two inequalities for worst-case expected total cost, completing that part of the theorem.
\end{proof}
For fixed positive scale parameters, $D_\eps$ is logarithmic in the inverse tolerance. Since $\max\{a,b\}\ge(a+b)/2$ for $a,b\ge0$, the theorem yields $\widetilde\Omega(\eps^{-3}+\gamma^{-3})$ total queries. Combining it with Theorem~\ref{thm:upper} establishes the rate in Corollary~\ref{cor:total}.

\section{Discussion and Future Directions}\label{sec:discussion}
Corollary~\ref{cor:total} characterizes the total query complexity for independent gradient and curvature tolerances. In particular, for any fixed $a>0$ and $\gamma=\eps^a$,
\begin{equation*}
    Q_{\mathrm{tot}}^\star(\eps,\eps^a)
    =\widetilde\Theta\left(\eps^{-3\max\{1,a\}}\right).
\end{equation*}
The result concerns total evaluations. The achievable pair in~\eqref{eq:Ng}--\eqref{eq:NH} and the marginal $\Omega(\eps^{-2})$ gradient lower bound in Appendix~\ref{app:gradient-test} do not characterize the full gradient--Hessian budget frontier. A Hessian response in the first-order hard family can also carry information about the next hidden block, so a total-call lower bound cannot be assigned entirely to gradients. Determining that two-resource frontier remains a separate problem.
The logarithmic factor in the first-order lower bound arises from the block size and the slow coordinate scale. Removing it would require a construction that simultaneously retains persistent-handle information control and uniform sample Hessian regularity at a smaller scaling cost. Fixed-dimensional sharp lower bounds are also not established here. Finally, the complete-Hessian query convention leaves eigensolver and matrix-storage costs outside the theorem; an HVP-only or arithmetic-complexity characterization would require a different resource analysis.

\bibliographystyle{plainnat}
\bibliography{references}
\clearpage
\appendix
\section{Concentration and Smooth Localization Details}\label{app:technical}
\subsection{Conditional concentration for tree increments}
\begin{proof}[Proof of \cref{lem:accuracy}]
Condition on the history up to the determination of an edge's endpoints. The centered difference for one fresh sample has mean zero and norm at most
\begin{equation*}
2\min\{\sigma,L\norm{y_k-y_{p(k)}}\}.
\end{equation*}
The first bound follows by subtracting the two centered noise vectors or matrices, and the second by sample and population Lipschitzness. Matrix Hoeffding, with self-adjoint dilation in the vector case, shows that the prescribed batch achieves an edge error at most $\eta/(m+1)$ with probability at least $1-\delta/(T+1)$ after increasing $C$. The root is treated similarly. If the increment bound is zero, that increment is exact and needs no concentration argument. A union bound is valid for adaptive endpoints because each concentration bound was conditional on their already being fixed. Along a parent path the population differences telescope exactly. The root error and at most $m$ edge errors add to at most $\eta$ by the triangle inequality. No conditional independence of cached errors given the entire trajectory is being assumed.
\end{proof}
\subsection{Radial extension and bounded sample regularity}
\begin{proof}[Proof of \cref{lem:global}]
The wall is nonnegative, the cutoff is in $[0,1]$, and $\phi_U\in[-\Delta/2,0]$, so $f_U(0)=0$ and $f_U\ge-\Delta/2$. Differentiating the cutoff product gives
\begin{align*}
 \norm{\nabla f_U}&\le b_0+B_1\sqrt\gamma+C\Delta/R,\\
 \norm{\nabla^2f_U}&\le B_2\gamma+
 C(B_1\sqrt\gamma/R+\Delta/R^2+b_0/R),\\
 \operatorname{Lip}(\nabla^2f_U)&\le L_2/4+
 C(B_2\gamma/R+B_1\sqrt\gamma/R^2+\Delta/R^3+b_0/R^2).
\end{align*}
Thus the gradient bound is at most $G$, the Hessian norm at most $L_1$, and the Hessian Lipschitz constant at most $L_2$, after reducing the tolerance threshold. Both noise errors are zero. The wall grows linearly at infinity, so the infimum is finite.

If $r=\norm{x}\ge R$, its radial gradient component obeys
\begin{equation*}
\left\langle\nabla f_U(x),x/r\right\rangle
 =b_0+\chi(r/R)\langle\nabla\phi_U(x),x/r\rangle
   +\frac{\chi'(r/R)}R\phi_U(x)
 \ge b_0-B_1\sqrt\gamma\ge b_0/2.
\end{equation*}
The last cutoff term is nonnegative. Therefore no such point meets the gradient test when $\eps<b_0/2$.

For $r\le R/2$ the wall is zero and the original chain curvature proves the assertion. For $R/2<r<R$, let $j$ be the first incomplete coordinate, so $|u_j^\top x|\le h$. Write $e=x/r$, $c=\langle u_j,e\rangle$, and $w=u_j-ce$. Then $|c|\le2h/R=\cO(\gamma^2)$ and $w$ is tangent to the sphere at $x$. Since $u_j^\top\nabla^2\phi_Uu_j\le-8\gamma$,
\begin{equation*}
w^\top\nabla^2\phi_Uw\le-8\gamma+B_2\gamma(2|c|+c^2)\le-7\gamma
\end{equation*}
for small $\gamma$. The wall's Hessian in the tangent direction contributes at most
\begin{equation*}
\frac{W_R'(r)}r\norm{w}^2\le 2b_0/R\le\gamma.
\end{equation*}
As the cutoff is identically one in this region, $w^\top\nabla^2f_Uw\le-6\gamma$. Normalizing the nonzero vector $w$ still yields a Rayleigh quotient strictly below $-\gamma$. This excludes all incomplete-coordinate outputs inside the ball.
\end{proof}
\subsection{Third-order reverse stick-breaking estimate}
\begin{proof}[Proof of \cref{lem:thirdstick}]
First assume $0<b_i<1$. Introduce independent Bernoulli variables $B_i$ with parameters $b_i(x_i)$ and let
\begin{equation*}
L=\max(\{i:B_i=1\}\cup\{0\}).
\end{equation*}
Then $\Pr(L=i)=w_i$ and $N=\bbE_BV_L$. For a direction $u=(u_i)$, define
\begin{equation*}
s_i=\frac{B_i-b_i}{b_i(1-b_i)},\qquad
 S_u=\sum_i s_iDb_i[u_i].
\end{equation*}
The variables $s_i$ are independent and centered, with $\bbE s_i^2=\beta_i^{-1}$ and $\bbE|s_i|=2$. Consequently
\begin{equation*}
\bbE S_u^2\le \ell_1^2\norm{u}^2.
\end{equation*}
The likelihood ratios for the second and third directional derivatives of the product Bernoulli mass are
\begin{align*}
 K_{u,v}={}&\sum_i s_iD^2b_i[u_i,v_i]
 +\sum_{i\ne j}s_is_jDb_i[u_i]Db_j[v_j],\\
 H_{u,v,w}={}&\sum_i s_iD^3b_i[u_i,v_i,w_i]\\
 &+\sum_{i\ne j}s_is_j(
 D^2b_i[u_i,v_i]Db_j[w_j]
 +D^2b_i[u_i,w_i]Db_j[v_j]
 +D^2b_i[v_i,w_i]Db_j[u_j])\\
 &+\sum_{i,j,k\ \mathrm{distinct}}
 s_is_js_kDb_i[u_i]Db_j[v_j]Db_k[w_k].
\end{align*}
Products $\prod_{i\in I}s_i$ indexed by distinct subsets $I$ are orthogonal in $L_2$. Grouping ordered pairs or triples by their underlying subsets, and using Cauchy--Schwarz, gives for unit directions
\begin{equation*}
\bbE|K_{u,v}|\le C(\ell_2+\ell_1^2),\qquad
 \bbE|H_{u,v,w}|\le C(\ell_3+\ell_1\ell_2+\ell_1^3).
\end{equation*}
For example, the degree-two part of $H$ has squared $L_2$ norm at most a constant times
\begin{equation*}
\left(\sum_i\frac{D^2b_i[u_i,v_i]^2}{\beta_i}\right)
 \left(\sum_j\frac{Db_j[w_j]^2}{\beta_j}\right)
 \le \ell_2^2\ell_1^2,
\end{equation*}
and the degree-three part is bounded by $C\ell_1^6$. The degree-one third-derivative term is bounded in $L_1$ by
$2\ell_3\sum_i\norm{u_i}\norm{v_i}\norm{w_i}\le2\ell_3$.

Differentiating the finite expectation yields
\begin{align*}
 DN[u]={}&\bbE DV_L[u]+\bbE[V_LS_u],\\
 D^2N[u,v]={}&\bbE D^2V_L[u,v]
 +\bbE[DV_L[u]S_v+DV_L[v]S_u]+\bbE[V_LK_{u,v}],\\
 D^3N[u,v,w]={}&\bbE D^3V_L[u,v,w]
 +\bbE\!\left[\sum_{\rm sym}D^2V_L[u,v]S_w\right]\\
 &+\bbE\!\left[\sum_{\rm sym}DV_L[u]K_{v,w}\right]
 +\bbE[V_LH_{u,v,w}].
\end{align*}
The preceding moment estimates prove \eqref{eq:stick1}--\eqref{eq:stick3}; \eqref{eq:stick0} follows from convexity of the weights. Gates attaining zero or one are handled by replacing $b_i$ with $\delta+(1-2\delta)b_i$ and sending $\delta\downarrow0$. The weighted derivative assumptions are preserved under this regularization.
\end{proof}
\subsection{Slow-chain derivatives and the stationarity obstruction}
\begin{proof}[Proof of \cref{lem:slowgeom}]
The range $q_j\in[-2,2]$ and $0\le b\le1$ give $-20T\le h\le0$, while the regularizer is nonnegative and vanishes at zero, proving \eqref{eq:slow-gap}. Direct differentiation gives
\begin{equation*}
\norm{Dq_j}\le2\alpha/\sqrt D,
 \quad \norm{D^2q_j}\le4\alpha^2/D,
 \quad \norm{D^3q_j}\le C\alpha^3/D.
\end{equation*}
The regularizer in one block has Hessian norm at most $8\alpha^2/D$ and third derivative norm at most $C\alpha^3/D$. Every local term of $h$ involves only two neighboring blocks and has bounded derivatives on $[-2,2]^2$. Bounded-degree aggregation proves \eqref{eq:slow-hess}--\eqref{eq:slow-third}. Each block gradient has norm $\cO(A\alpha/\sqrt D)$, which proves \eqref{eq:slow-grad-global} after summing squares over the $T$ blocks.

For \eqref{eq:slow-stationarity}, let $j$ be the first block with $q_j\le1/2$ and write $B=\partial_jh(q)$. Its incoming contribution is at most $-2$, while its outgoing contribution is nonpositive, so $B\le-2$. Set $c=-B\ge2$ and
$r_\ell=\theta_{j\ell}\tanh(\alpha x_{j\ell})\in[-1,1]$. Up to multiplication by the sign $\theta_{j\ell}$, the $\ell$th coordinate of the block gradient equals
\begin{equation*}
\frac{A\alpha}{D}\,z_c(r_\ell),
 \qquad
 z_c(r)=2(cr^2+4r-c).
\end{equation*}
The positive root $r_c=(\sqrt{c^2+4}-2)/c$ is increasing for $c\ge2$ and satisfies $r_c\ge\sqrt2-1$. The other root is below $-1$, and the elementary factorization gives $|z_c(r)|\ge4|r-r_c|$ on $[-1,1]$. Hence
\begin{align*}
 \norm{\nabla_j\Phi_\Theta(x)}
 &\ge\frac{A\alpha}{D\sqrt D}\sum_{\ell}|z_c(r_\ell)|\\
 &\ge\frac{4A\alpha}{\sqrt D}
 \left(r_c-\frac1D\sum_\ell r_\ell\right)
 \ge\frac{4A\alpha}{\sqrt D}(\sqrt2-1-1/4).
\end{align*}
Since $D^{-1}\sum r_\ell=q_j/2\le1/4$, this proves the claim.
\end{proof}
\section{Fixed-Budget Implementation}\label{app:caps}
Take
\begin{equation*}
T=1+\left\lceil48\Delta(d_g^{-1}+d_H^{-1})\right\rceil.
\end{equation*}
This is the iteration cap in Theorem~\ref{thm:upper}, using the values of $d_g$ and $d_H$ in Section~\ref{sec:upper-proof}. Use Lemma~\ref{lem:accuracy} with failure probability $1/96$ for each tree, so their combined failure probability is at most $1/48$. The event that the real run reaches $T$ updates without returning and without an inaccurate estimator has probability at most $1/48$: on this event the shadow has $T$ updates, while its expected count is at most $\Delta(d_g^{-1}+d_H^{-1})$.

Let $\mathcal L=1+\log(C(d+1)(T+1))$. Lemma~\ref{lem:charge} bounds the two expected shadow costs by $C\mathcal L^4B_g$ and $C\mathcal L^4B_H$. Impose the integer caps $Q_g=\lceil48C\mathcal L^4B_g\rceil$ and $Q_H=\lceil48C\mathcal L^4B_H\rceil$ from Theorem~\ref{thm:upper}, aborting before any query that would exceed a cap. Each cap-overflow event on an otherwise accurate trajectory has probability at most $1/48$ by Markov's inequality and coupling with the shadow. This does not apply an expected-cost assertion to a bad-event trajectory. A union bound over estimator failure, iteration timeout, and the two budget overflows leaves success probability at least $1-4/48>2/3$.

\section{A Marginal Gradient Lower Bound}\label{app:gradient-test}
Let $b=\sigma_g/2$ and take sufficiently small $\eps$ so that $2\eps\le b/2$. For $\theta\in\{-2\eps,+2\eps\}$, let $Z\in\{-b,b\}$ satisfy $\bbE_\theta Z=\theta$. Use opaque sample handles for $Z$, and define in dimension one
\begin{equation*}
F_\theta(x;Z)=4\eps\sqrt{1+x^2}+Zx,\qquad
 f_\theta(x)=4\eps\sqrt{1+x^2}+\theta x.
\end{equation*}
Sample derivatives satisfy
\begin{equation*}
F_\theta'(x;Z)=\frac{4\eps x}{\sqrt{1+x^2}}+Z,\qquad
 F_\theta''(x;Z)=\frac{4\eps}{(1+x^2)^{3/2}}.
\end{equation*}
The Hessian is independent of $Z$ and of the hypothesis. The gradient noise is at most $b+2\eps\le\sigma_g$. All sample smoothness constants are $\cO(\eps)$, the population gradient is at most $6\eps$, and
\begin{equation*}
f_\theta(0)-\inf f_\theta=(4-2\sqrt3)\eps.
\end{equation*}
Thus these are admissible for the fixed positive model parameters when $\eps$ is small. The population Hessian is nonnegative everywhere, so the curvature condition is automatic for every $\gamma>0$.

For $\theta=+2\eps$, an $\eps$-stationary point must have $x<0$; for $\theta=-2\eps$ it must have $x>0$. Thus a successful optimization algorithm also distinguishes the two hypotheses with success at least $2/3$.

An initial gradient evaluation on a new handle reveals exactly one $Z$. Repeated evaluations on that handle reveal no additional random information, and any number of Hessian queries reveal none. The per-handle Bernoulli KL divergence between the two hypotheses is $\cO(\eps^2/b^2)$. By the chain rule for the adaptive transcript, its KL divergence is at most $C\eps^2\bbE_+N_g/b^2$; reuses contribute zero conditional divergence. Distinguishability with success at least $2/3$ under both hypotheses requires total variation at least $1/3$, and hence KL divergence bounded below by a numerical constant. Therefore
\begin{equation*}
\boxed{\bbE_+N_g=\Omega(b^2/\eps^2)=\Omega(\eps^{-2}).}
\end{equation*}
The same argument applies to fixed budgets. Padding the construction by unused coordinates embeds it into the lower-bound dimension of the previous section.

\end{document}